\documentclass{amsart}   
\usepackage{ulem}
\usepackage{graphicx, amsmath, amssymb, amsthm, amscd, pgfplots,pstricks-add}
\usepackage{enumerate}
\usepackage{color}
\usepackage{mathrsfs}
\usepackage{epsf}
\usepackage{pgfplots}
\usepackage{hyperref}
\pgfplotsset{compat=1.15}
\usepackage{mathrsfs}
\usetikzlibrary{arrows}
\definecolor{ffqqqq}{rgb}{1,0,0}
\usepackage[T1]{fontenc}

\newtheorem{thm}{Theorem}[section]
\newtheorem{cor}[thm]{Corollary}
\newtheorem{lem}[thm]{Lemma}

\newtheorem{que}[thm]{Question}

\theoremstyle{definition}
\newtheorem{defn}[thm]{Definition}
\newtheorem{rem}[thm]{Remark}

\newcommand{\N}{\mathbb{N}}

\newcommand{\G}{\mathcal{G}}
\newcommand{\Z}{\mathbb{Z}}

\newcommand{\R}{\mathbb{R}}
\newcommand{\F}{\mathcal{F}}
\renewcommand{\H}{\mathcal{H}}
\usepackage{mathrsfs}

\newcommand{\eps}{\varepsilon}

\newcommand{\cl}[1]{\overline{#1}}
\newcommand{\bd}[1]{\partial (#1)}
\newcommand{\Int}[1]{ \textrm{Int} (#1)}

\newcommand{\End}{\mathsf{End}}

\newcommand{\diam}{\text{diam\,}}

\newcommand{\mesh}{\text{mesh\,}}

\newtheorem*{claim*}{Claim}

\numberwithin{equation}{section}

	\title{Every nondegenerate Peano continuum admits a pure mixing selfmap}

	\author[K.\ Karasova]{Klara Karasova}
	\author[M.\ Kowalewski]{Micha\l{}~Kowalewski}
	\author[P.\ Oprocha]{Piotr Oprocha}

\address[K. Karasova]{Charles University, Faculty of Mathematics and Physics
}
\email{karasova@karlin.mff.cuni.cz}
\address[M. Kowalewski]{AGH University of Krakow, Faculty of Applied Mathematics,
al.\ Mickiewicza 30, 30-059 Krak\'ow, Poland.}
\email{kowalewski@agh.edu.pl}

 \address[P.\ Oprocha]{
Centre of Excellence IT4Innovations - Institute for Research and Applications of Fuzzy Modeling, University of Ostrava, 30. dubna 22, 701 03 Ostrava 1, Czech Republic}
\email{piotr.oprocha@osu.cz}

\begin{document}
\begin{abstract}
We prove that every Peano continuum (a space that is a continuous image of $[0,1]$) admits a topologically mixing but not exact map. The constructed map has a dense set of periodic points. 
\end{abstract}


\maketitle

\section{Introduction}
In the past, several attempts have been made to formalize mixing in a mathematically rigorous way. Topological mixing represents a natural expectation connected with this process: after some time, the image of any {nonempty} open set $U$ starts to intersect with any other {nonempty} open set $V$, and this nonempty intersection never disappears. It is a stronger property than transitivity but often weaker than topological exactness (every open set eventually covers the whole space under iteration) or the specification property (for more details, see survey \cite{KLO}). Transitivity (or stronger mixing properties) can also be used to formalize the mathematical understanding of chaos. This idea originated from the classical book of Devaney \cite{Dev} and was later adopted and extended by various authors (cf. \cite{KMis} and references therein). It was also observed that, in various situations, dense periodicity is a direct consequence of transitivity (see \cite{DirSno}).
This is perhaps the simplest example of how topological structure can influence dynamical complexity, motivating  the question of the existence of maps with prescribed chaotic behavior in a specified class of spaces.

Classical results already highlight the problem. For example, every weakly mixing map of the interval is necessarily mixing (see \cite{Ruette}). The Warsaw circle illustrates another restriction: it admits a mixing map yet, it is easy to see that no proper subset covers the entire space in finite time (hence no exact map exists, see e.g. \cite{sivak}). 
Properly placed ``obstacles'' can significantly determine the dynamics or limit the set of possible homeomorphisms or surjections (see \cite{DowSno}, cf. \cite{Bor}). It is also reflected in studies on the relationships between the values of topological entropy, the structure of the space, and the richness of the dynamics. It was first observed in \cite{KH} that the entropy of mixing circle maps can be arbitrarily small, while in mixing maps which are not exact (so-called \textit{pure mixing} maps) this infimum is bounded away from zero (and this phenomenon remains valid on topological graphs \cite{HKO}). Recently, \v{S}pitalsk\'{y} noted in \cite{Vlad} that this infimum is zero in exact maps on dendrites, provided that no subtree of the dendrite contains all {free} arcs  {in the dendrite} (see \cite{KOT}). It is also worth noting that some spaces, while allowing mixing maps, make the admissible values of entropy very rigid, with the Knaster continuum serving as a particular example of this case (e.g., see \cite{BruSti}).
It is also clear that several other obstacles may arise, for example, spaces such as $[0,1]$ allow topological mixing, but homeomorphisms have only simple dynamics. It also intuitively seems that mixing, but not exact maps, may require quite complicated dynamics, as is evident in interval maps already. It is also worth mentioning the recent paper of Illanes and Rito \cite{MixingDendrites}, where the authors construct a strongly mixing map (every point of the space is eventually covered by any open set) that is not exact (cf. \cite{HKO}).
{Recently, in \cite{karasova2025chaospeanocontinua}
it was shown that every continuum that is a continuous image of the interval $[0,1]$ (the so-called Peano continuum) admits an exact map.} This generalizes the finite-dimensional result of  \cite{AgronskyCeder}. Naturally, this raises the question of whether mixing may exist on every Peano continuum without exactness.
The answer is evidently positive in some cases, since there are well known examples of homeomorphism on the torus or the previously mentioned examples on the interval.
The main result of this paper is the following:
\begin{thm}\label{thm:pure}
    Every nondegenerate Peano continuum admits a pure mixing (i.e., mixing but not exact) surjection with a dense set of periodic points.
\end{thm}

As we show, our construction, unfortunately, usually leads to large entropy (see Remark~\ref{rem:pure}; cf. Corollary \ref{cor:pure}) and, by the  methods involved, the constructed map is never invertible. It would be good to know what the infimum of topological entropy is on a given Peano continuum, or if such a continuum admits a mixing homeomorphism, but we do not have any answers to either of these questions. It will need further research and other methodologies.

\section{Preliminaries}
 We denote by $\Z$, $\N$, and $\N_0$ the sets of integers, positive integers, and nonnegative integers, respectively.
By a \textit{continuum, }we mean a compact, connected, metrizable space.  For a metrizable space $X$, we always fix a metric $d_X$ compatible with the topology of $X$. If the space $X$ is fixed and there is no ambiguity, we will simply write $d$ for any choice of a compatible metric. For a set $A \subset X$, we denote its interior by $\Int A$, its closure by $\overline{A}$, and its boundary by $\partial A$. For a metric space $(X,d)$ and  any point $x \in X$, we denote the ${\eps}$-ball centered in $x$ by $B(x,\varepsilon):=\{y; d(x,y)<\varepsilon \}$. For a set $U \subset X$, we write $\mathcal{N}(U,{\eps})=\bigcup_{x \in U}B(x,{\eps})$.  

We say that $x \in X$ is an endpoint of a continuum $X$ if arbitrarily small neighborhoods of $x$ have a boundary consisting of one point. We denote $\End(X)$ for the set of endpoints of the continuum $X$. We write $\diam A=\sup_{x,y \in A} d(x,y)$ for the diameter of the {nonempty} set $A \subset X$. For any   finite cover $\mathcal{F}$ of a continuum $X$, we write $\mesh \mathcal{F}=\max_{F \in \mathcal{F}} \diam F $

\subsection{Dynamical systems}
A \textit{dynamical system} is a pair $(X,T)$ consisting of a compact metrizable space $X$ and a continuous map $T\colon X\to X$. 
We denote $T^0=id_{X}$ and $T^{n+1}=T \circ T^n$ for $n \geq 0$. We say that $x \in X$ is a \textit{fixed point (of T)} if $T(x)=x$, and we say that $x$ is a  \textit{periodic point (of T)} if $T^n(x)={x}$ for some $n >0$. The dynamical system $(X,T)$ (or just map $T$) is said to be \textit{transitive} if, for any nonempty open subsets $U,V \subset X$, there exists a natural number $n\in \mathbb{N}$ such that  $f^n(U) \cap V \neq \emptyset$. If a transitive system has a dense set of periodic points, we call it \textit{Devaney chaotic}. We say that a dynamical system $(X,T)$ (or map $T$) is \textit{mixing} if, for any two nonempty open subsets $U,V \subset X$, there exists a natural number $n_0 \in \mathbb{N}$ such that for every $n \geq n_0$, we have $f^n(U) \cap V \neq \emptyset$. Equivalently, the system is mixing if, for any nonempty open set $U \subset X$, the sequence of sets $f^n(U)$ converges to $X$ in the Hausdorff metric. We say that a dynamical system $(X,T)$ (or map $T$) is \textit{ exact} or \textit{locally eventually onto (leo) }if, for any nonempty open set $U$, there is a natural number $n \in \mathbb{N}$ such that $f^n(U)=X$. It is easy to see that every exact map is mixing and that every mixing map is transitive. 

\begin{defn}
We say that a dynamical system $(X,T)$ (or the map $T$) is \textit{pure mixing} if it is mixing but not exact.    
\end{defn}

Let $(X,T)$ be a dynamical system, and let ${\eps}>0$. We say that a set $E \subset X$ is $(n,{\eps})${-}separated if, for any two distinct $x,y$, the inequality $\max_{0 \le i \le n} d(f^i(x),f^i(y))\geq {\eps}$ holds. Denote $s(n,\eps)$ to be the maximal cardinality of a $(n,{\eps})${-}separated set. We call the number $h_{top}(T)=\lim\limits_{\eps \to 0}\limsup\limits_{n \to \infty} \frac1n s(n,{\eps})$ the \textit{topological entropy} of the map $T$ (or of the system $(X,T)$). 

\subsection{Peano continua}
    For a topological space  $X$ and a point  $x\in X$, the \textit{Menger order} of $x$ in $X$ is the least natural number $n$ such that there is a basis of the neighborhood system of $x$ formed by sets whose boundaries are of size at most $n$, if any such natural number exists, and $\infty$ otherwise.

    A metric space $(X,d)$ is called a \textit{Peano space }provided that for each $p \in X$ and each neighborhood $N$ of $p${,} there is a connected open subset $U$ of $X$ such that $p \in U \subset N.$ A \textit{Peano continuum} is a Peano space that is a continuum. 

    A nonempty subset $Y$ of a metric space $(X,d)$ is said to have \textit{the property $S$} provided that for each $\eps > 0,$ there is $\F$ a finite cover of $Y$ formed by connected sets such that $\mesh\F<\varepsilon$.

\begin{thm}\label{finecovers}
    Let $X$ be a Peano continuum and $\varepsilon>0$. There is a finite cover $\F$  of $X$ formed by Peano continua such that $F=\overline{\Int F}$ for every $F\in\F$ and $\mesh\F<\varepsilon$.
\end{thm}

\begin{proof}
    It is easily seen that $X$ has the property S, hence, by \cite[8.9]{nadler}, there is $\G$ a finite cover of $X$ formed by open sets with the property S satisfying $\mesh\G<\varepsilon$. Let $\F:=\{\cl{G};\,G\in\G\}$, clearly $\F$ is a finite cover of $X$ by continua, $\mesh\F=\mesh\G<\varepsilon$ and every member $F\in F$ satisfies $F=\overline{\Int F}$. By \cite[8.5]{nadler}, each member of $\F$ has the property S and thus is a Peano continuum by \cite[8.4]{nadler}.
\end{proof}

We will need the following important result {\cite[Theorem~8.19]{nadler}}:
\begin{thm}\label{nadler:8.19}
    For every nondegenerate continuum $X$, every Peano continuum $Y$, distinct points $x_1,\dots,x_n\in X$, and any points $y_1,\dots,y_n\in Y$, there is a continuous surjection $f\colon X\to Y$ such that $f(x_i)=y_i$ for every $1\leq i\leq n$.
\end{thm}

By definition, every point in a Peano continuum has an arbitrarily small connected neighborhood. This immediately leads to the following.

\begin{rem}\label{opencomponents}
    Connected components of open sets are open in every Peano continuum. 
\end{rem}

In our constructions of maps on Peano continua, we will strongly rely on the following extension theorem.

\begin{thm}[Tietze Extension Theorem]\label{tietze}
    Let $X$ be a normal topological space and $A\subset X$ a closed set. Assume that $f:A\to [0,1]$ (alternatively, $f:A\to \R$) is continuous. Then $f$ can be continuously extended over $X$, that is, there exists a continuous function $\bar{f}:X\to [0,1]$ (alternatively, $\bar{f}:X\to \R$) satisfying $\bar{f}(x)=f(x)$ for every $x\in A$.
\end{thm}

\section{Cut points in Peano continua}

In what follows, we will need a good understanding of the structure of cut points of Peano continua and their influence on other aspects of the space. {Recall that a point $x \in X$ of a connected set $X$ is called a \textit{cut point} of $X$ if $X \setminus \{x\}$ is not connected. We say that a point is a \textit{non-cut} point if it is not a cut point. We call $x \in X$ a \textit{local cut point}, if there is some connected neighborhood $U \subset X$ such that $x$ is a cut point of $U$.}

\begin{lem}\label{densecutpoints}
    Let $X$ be a nondegenerate Peano continuum such that the set of all cut points of $X$ is dense in $X$. Then $\End(X)\ne\emptyset$.
\end{lem}

\begin{proof}
   The proof follows the ideas in \cite[Theorem 6.6]{nadler}  (non-cut point existence). Denote
   \[
   \F':= \{F\subset X:\,F \text{ is a continuum, } |\bd F|\leq1, F\setminus\bd F \text{ is a component of }X\setminus \bd F\}
   \]
    and note that $\F'$ is nonempty since $X\in \F'$. We say that a set $\F\subset\F'$ is a nest if any two $V,W
   \in \F$ satisfy $V\subset W$ or $W\subset V$. It is clear that nests are partially ordered by inclusion and that every chain of nests has an upper bound, which is the nest obtained as the union of the elements in the chain. Hence, by the Kuratowski-Zorn Lemma, there is a maximal nest $\hat{\F}\subset\F'$. Denote $F:=\bigcap\hat{\F}$.

   We claim that by the compactness of $X$, if $U$ is an open set satisfying $F\subset U$, then there exists $\hat{F}\in\hat{\F}$ satisfying $\hat{F}\subset U$. Indeed, if $F\subset U$ for some open $U$, then $\{X\setminus \hat{F};\,\hat{F}\in\hat{\F}\}\cup \{U\}$ is a cover of $X$ by open sets. Thus, there are $\hat{F}_1,\dots,\hat{F}_n\in\hat{\F}$ such that $X=(X\setminus\hat{F}_1)\cup\dots\cup(X\setminus\hat{F}_n)\cup U$ by the compactness of $X$. Since $\hat{\F}$ is a nest, there is $1\leq i\leq n$ such that $\hat{F}_i=\hat{F}_1\cap\dots\cap\hat{F}_n$; in other words, $(X\setminus\hat{F}_i)=(X\setminus\hat{F}_1)\cup\dots\cup(X\setminus\hat{F}_n)$. Hence $X=(X\setminus\hat{F}_i)\cup U$,  which implies that $\hat{F}_i\subset U$.

   Using the claim, we obtain that $F\ne\emptyset$, since otherwise for $U:=\emptyset$, there would be $\hat{F}\in\hat{\F}$ satisfying $\hat{F}\subset U=\emptyset$, which is a contradiction. Using the claim for $U:= \mathcal{N}(F,1/n)$, $n\in\N$, we obtain that there are $F_n\in\hat{\F}$ such that $F_n\subset \mathcal{N}(F,1/n)$, $n\in\N$. Thus, in particular, $\bigcap_{n\in\N}F_n=F$. For each $n\in\N$, let $p_n\in X$ be the point satisfying $\bd{F_n}=\{p_n\}$ if $|\bd{F_n}|=1$, respectively let $p_n\in X$ be arbitrary if $F_n=X$. After passing to a subsequence, if necessary, we may assume that $\lim p_n=p$ exists by the compactness of $X$, while still having $\bigcap_{n\in\N}F_n=F$.

   We want to prove that $|F|=1$, so assume for the sake of contradiction that $F$ is nondegenerate. Thus, there exists $s\in F\setminus\{p\}$ since $F$ has at least two points. Let $V$ be a neighborhood of $s$ that is disjoint with a neighborhood of $p$, we may assume that $V$ is connected since $X$ is Peano. We may also assume that $p_n\not\in V$ for every $n\in\N$ after we drop finitely many elements of the sequence, still having $\bigcap_{n\in\N}F_n=F$. Fix $n\in\N$ and note that $V$ is a connected subset of $X\setminus\{p_n\}\subset X\setminus\bd{F_n}$, thus, we may denote by $C_n$ the connected component of $V$ in $ X\setminus\bd{F_n}$. Further, since $F_n\setminus\bd{F_n}$ is a connected component of $X\setminus\bd{F_n}$ too and $s\in F\cap V\subset C_n\cap (F_n\setminus\bd{F_n})$, we obtain $C_n=F_n\setminus\bd{F_n}$. Thus, we obtain $V\subset F_n\setminus\bd{F_n}$ for every $n\in\N$, in particular, $V\subset F$.

   {Let $q \in V$ be a cut point of $X$. By Remark~\ref{opencomponents}, for each component $W$ of $X \setminus \{q\}$, $\overline{W}= W \cup \{q\} \in \F'$. Let $W_0$ and $V_0$ be components of $X \setminus \{q\}$ such that $p \in W_0$ and $V_0 \neq W_0$. Since $q \in \overline{W}_0 \cap \overline{V}_0\cap V$, we can fix points $s_0 \in V_0 \cap V $ and $w_0 \in W_0 \cap V $. 
    Then the set $V_0$ is a connected neighborhood of $s_0\in V\subset F$ and also $V_0$ is disjoint with $W_0$ a neighborhood of $p$, so proceeding as in the previous paragraph, we conclude that $V_0 \subset F$.
   Thus $F_0 := \overline{V_0} = V_0 \cup \{q\}\subset F$. Since $w_0  \in F \setminus F_0$, we have that $F_0 \subsetneq F$. Thus $F_0 \subsetneq F^*$ for every $F^* \in \hat{\F}$, hence $\hat{\F} \cup \{F_0\} \subset \F'$ is a nest strictly larger than $\hat{\F}$, which is a contradiction. Thus
$|F|= 1$. 

Let $x  \in X$ be the unique point satisfying $F= \{x\}$. Observe that $x \in \End(X)$, since $\{F_n\}_{n\in \N}$ is a base of neighborhoods for $x$ with singleton boundaries, which completes the proof.}
\end{proof}

\begin{lem}\label{existencenonloc}
    Let $X$ be a nondegenerate Peano continuum. Then there is $x \in X$ satisfying one of the following:
    \begin{enumerate}
        \item\label{notlocalcutpoint} $x$ is not a local cut point of $X$, or,
        \item\label{localcutpointordertwo} $x$ is a local cut point of $X$ that is not a cut point of $X$, and the Menger order of $x$ is equal to 2.
    \end{enumerate}
\end{lem}

\begin{proof}

If there is $x\in X$ satisfying \eqref{localcutpointordertwo}, then the proof is complete, so let us assume that such a point does not exist.
    By Lemma \ref{densecutpoints} , if the set of cut points of $X$ is dense in $X$, then 
    $\End(X)\neq \emptyset$. Fix $x\in \End(X)$ and let $U\ni x$ be an {open} connected neighborhood (which exists because $X$ is a Peano continuum). Since $x\in\End(X)$, we also have $x\in\End(U)$, and thus $x$ is a non-cut point of $U$ (cf. \cite[Exercise 6.25b)]{nadler}).

Since $U$ was arbitrary, \eqref{notlocalcutpoint} holds.
    
      Consider the second possibility: that the set of cut points of $X$ is not dense.  
    Let $U\subset X$ be a nonempty connected open set
    that contains no cut points of $X$. Put $U_0:=U$ and proceed by induction.
    
    Assume that for some $n\geq 0$ we have a sequence
    \[
    U_{n}\subset F_n \subset U_{n-1}\subset F_{n-1}\subset\dots\subset F_1\subset U_0,
    \]
    where each $U_i$ is open, connected, and nonempty, and each $F_i$ is a Peano continuum satisfying $\diam F_i<2^{-i}$. Furthermore, assume that for each $1\leq i\leq n$, at most countably many cut points of $F_i$ belong to $U_i$, and that we have enumerated all such points as $x^i_1,x^i_2,\dots$. Finally, assume that $x^i_k\not\in U_j$ if $\max\{i,k\}\leq j$.

      By Theorem \ref{finecovers}, there is a Peano continuum $F_{n+1}\subset U_n$ of diameter $\diam (U_n)<2^{-n-1}$ and with a nonempty interior. Take $U'_{n+1}\subset \Int{F_{n+1}}\ne\emptyset$ any open and connected nonempty set.
    
    First, assume that there is $x\in U_{n+1}'$ a cut point of $F_{n+1}$ such that the Menger order of $x$ is 2. Then it is easy to observe that $x$ satisfies \eqref{localcutpointordertwo}: $x$ is a local cut point of $X$ since it is a cut point of the open connected set $U'_{n+1}\subset F_{n+1}$, $x$ is not a cut point of $X$ since no point in $U_0$ is and $x\in U_0$, and the Menger order of $x$ is the same in both $X$ and $F_{n+1}$ since the Menger order is a local notion. This is a contradiction.
    
     Thus, assume that no cut point of $F_{n+1}$ contained in $U_{n+1}'$ has the Menger order 2. Then, by \cite[Theorem 7]{Whyburn} the set of cut points of $F_{n+1}$ that belong to $U_{n+1}'$ is at most countable. Let us enumerate all such points as $x^{n+1}_1,x^{n+1}_2,\dots$. Choose $U_{n+1}\subset U'_{n+1}$ any nonempty open connected set satisfying $x^i_{k} \notin U_{n+1}$ for $i,k\leq n+1$. This completes the induction.

 By the compactness of $X$, there is $x\in X$ such that $\bigcap_{n\in\N}F_n=\{x\}$. It remains to check that $x$ is not a local cut point of $X$. Thus, assume for the sake of contradiction that $x$ is a local cut point of $X$. Then there is $V$, a connected neighborhood of $x$ such that $V\setminus\{x\}$ is not connected. By the construction, there is $n\in\N$ such that $F_n\subset V$. Then $x\in F_{n+1}\subset U_n\subset F_n$, hence $F_n$ is a neighborhood of $x$. Necessarily, $x\in U_n$ must be a cut point of $F_n$ and thus $x=x^n_i$ for some $i\in\N$. However, by the construction, 
    \[
    x=x^n_i\notin U_{\max\{n,i\}}\supset F_{\max\{n,i\}+1}\ni x,
    \]
    a contradiction. Thus $x$ satisfies $\eqref{notlocalcutpoint}$.
\end{proof}

\begin{lem}\label{lifting}
    Suppose $X$ is a Peano continuum{,} $x\in X$ has Menger order 2, is a local cut point, but is not a cut point. Then there is $\widehat{X}$ a Peano continuum 
     and a continuous map $\phi\colon \hat{X}\to X$ such that $\phi^{-1}(y)$ is a singleton for every $y\neq x$ and consists of exactly two endpoints otherwise.
\end{lem}

\begin{proof}
    Let $U$ be a connected neighborhood of $x$ such that $U\setminus \{x\}$ is not connected. Cover $X$ with a finite family $\F$ consisting of Peano continua small enough to ensure that if $F\in\F$, $x\in F$, then $F\subset U$. Denote $A:=\bigcup \{F\in \F;\,x\in F\}$ and observe that $A\subset U$ is a connected neighborhood of $x$, in particular, $x$ is a cut point of $A$ of Menger order 2. It follows easily using \cite[Observation 13]{KarasovaVejnar} that $A\setminus\{x\}$ has exactly two components, say $A_1,A_2$ (at least two components since $x$ is a cut point of $A$, and at most two components by the Observation, since the Menger order of $x$ in $A$ is 2).
    
     By definition, $A$ is a Peano continuum and $\bd{A_1}=\{x\}=\bd{A_2}$.
    Hence $A_i\cup\{x\}$, $i=1,2$ are Peano continua by \cite[Observation 11]{KarasovaVejnar}, as $A_1,A_2$ are open in $A$ by Remark~\ref{opencomponents}. Replace each $A_i\cup\{x\}$ by  the one point compactification $A_i\cup\{x_i\}$, where points $x_1,x_2$ are distinct and denote this extended space $\widehat{X}$. It is clear that $\widehat{X}$ is a Peano continuum, as a connected union of finitely many Peano continua (recall that $x$ is not a cut point of $X$). Since the natural identification $\phi \colon \widehat{X}\setminus \{x_1,x_2\} \to X\setminus \{x\}$ is a homeomorphism, it extends to the desired continuous map $\phi$ on $\widehat{X}$.
\end{proof}

Informally speaking, Lemma~\ref{lifting} ensures that we can ``cut'' $X$ at the local cut point $x$, keeping it a Peano continuum.

\begin{lem}\label{noncutproperty}
    Let $X$ be a Peano continuum, $x \in X$ a non-cut point, and ${\eps}>0$. Then there are Peano continua $A, B \subset X$ such that $x \in X \setminus B\subset A\subset B(x,\varepsilon)$, and $\cl{{\Int B}}=B$.
\end{lem}
\begin{proof}
    Let $\F$ be a finite cover of $X$ by Peano continua such that $\mesh\F<\varepsilon/2$. Let $\mathcal{A}=\{F\in\F:\,x\in F\}$ and $\mathcal{B}=\{F\in\F:\,x\notin F\}=\F\setminus \mathcal{A}$. The set $X\setminus \{x\}$ is an open, connected subset of the Peano continuum $X$, hence it is arcwise connected (see \cite[Theorem 8.26]{nadler}). Therefore, for any two sets $B,C \in \mathcal{B}$, we can find an arc $L_{B,C} \subset X \setminus \{x\}$ such that $B \cap L_{B,C} \neq \emptyset$ and $C \cap L_{B,C} \neq \emptyset$.

      Now, we let $A=\bigcup \mathcal{A}$ and $B'=\bigcup\mathcal{B} \cup \bigcup_{B,C \in \mathcal{B}} L_{B,C}$. Clearly, both $A,B'$ are Peano continua as connected unions of Peano continua. By definition $x \in X \setminus B'\subset A\subset B(x,\varepsilon)$. We will slightly modify $B'$ to obtain a set such that $\cl{{\Int B} }=B$. Take $\delta>0$ sufficiently small, so that $B(x,\delta)\subset A\setminus B'=X\setminus B'$. Let $\H$ be a finite cover of $X$ by Peano continua such that $\mesh\H<\delta/2$ and $\cl{{\Int H} }=H$ for every $H\in\H$ provided by Theorem~\ref{finecovers}. Let
    \[
    B:=\bigcup\{H\in\H;\,H\cap B'\ne\emptyset\}.
    \]
    It is clear that $\cl{{\Int B}}=B$ and
    $$
    x\in X\setminus B\subset X\setminus B'\subset A\subset B(x,\eps)
    $$
    completing the proof.
\end{proof}

\section{Proof of Theorem~\ref{thm:pure}}

The proof of Theorem~\ref{thm:pure} will be obtained by a careful construction of two maps $f\colon X\to [0,1]$ and $g\colon [0,1]\to X$ whose composition $h=g\circ f$ provides the desired map on the Peano continuum $X$. The first of these maps will be constructed with the help of the following general tool.

{
\begin{lem}\label{projections}
    Let $X$ be a metric space, $x,y  \in X$, $x \neq y$ and let $T$ be a countable subset of $X$ with $\{x,y\} \cap T= \emptyset$. Then there exists a continuous function $f\colon X \to [0,1]$ such that $f^{-1}(0) = \{x\}$, $ f^{-1}(1) = \{y\}$, and $f|_T \colon  T  \to [0,1]$ is one-to-one. In particular, if $X$ is a Peano continuum,
then we can choose $f \colon  X \to [0,1]$ to satisfy that the image of any nonempty open set has a nonempty interior.
\end{lem}
\begin{proof}
Let $T = \{t_1, t_2, \dots\}$ and $t_i \neq t_j$ when $i \neq j$. Denote $T_n = \{t_1, \dots, t_n\}$ for $n \in \mathbb{N}$, $T_0=\emptyset$. Further, let $f_0 \colon X \to [0,1]$ be given by $f_0(p) = \frac{d(p,x)}{d(p,x)+d(p,y)}$ and note that $f_0^{-1}(0) = \{x\}$, $f_0^{-1}(1) = \{y\}$. Denote $F:=f_0^{-1}([0,1/4])$, $E:=f_0^{-1}([3/4,1])$.

We will inductively construct a sequence of continuous maps $f_n \colon X \to [0,1]$ such that:
\begin{enumerate}
    \item\label{projections:p:1} for all $n \in \N_0: f_n^{-1}(0) = \{x\}$ and $f_n^{-1}(1) = \{y\}$
    \item\label{projections:p:5} for all $n \in \N_0$: $f_{n+1}|_F \geq f_n|_F$ and $f_{n+1}|_E\leq f_n|_E$,
    \item\label{projections:p:2} for all  $n \in \N $ the map $f_{n}|_{T_n}$ is injective,
    \item\label{projections:p:3} for all $n \in \N_0$: $f_{n+1}|_{T_n} = f_n|_{T_n}$,
    \item\label{projections:p:4} for all $n \in \N_0$: $d(f_n, f_{n+1}) < 2^{-n-4}$.
\end{enumerate}

It is immediate that the hypotheses for $n=0$ are satisfied.

Thus assume that $f_n$ is already defined and we will construct $f_{n+1}$. If $f_n(t_{n+1}) \notin f_n(T_n)$, we put $f_{n+1} := f_n$.  

Suppose that $f_n(t_{n+1}) \in f_n(T_n)$. Choose a neighborhood $G \subset X\setminus (T_n\cup\{x,y\})$  of the point $t_{n+1}$ such that if $t_{n+1}\notin F$, then $G\cap F=\emptyset$ and if $t_{n+1}\notin E$, then $G\cap E=\emptyset$, this is possible since the sets $F,E$ are closed. Further,  fix $0<\eps<2^{-n-4}$ satisfying 
\[
(f_n(t_{n+1})-\eps,f_n(t_{n+1})+\eps)\cap f_n(T_n)=\{f_n(t_{n+1})\}.
\]
If $f_n(t_{n+1})\leq 1/2$ we put
$$
f_{n+1}=f_n+\min\{\eps,d(\,\cdot\,,X\setminus G)\},
$$ 
otherwise we put
$$
f_{n+1}=f_n-\min\{\eps,d(\,\cdot\,,X\setminus G)\}.
$$

We will only verify \eqref{projections:p:5}, as verifying the other conditions is straightforward. To see that $f_{n+1}|_F \geq f_n|_F$, first note that $f_{n+1}|_F = f_n|_F$ whenever $F\subset X\setminus G$, in particular whenever $t_{n+1}\notin F$ by our choices. Thus assume $t_{n+1}\in F$, in other words $f_0(t_{n+1})\leq 1/4$. Hence $f_n(t_{n+1})< f_0(t_{n+1})+ 2^{-4}+\dots + 2^{-n-4}<f_0(t_{n+1})+ 2^{-3}$ by \eqref{projections:p:4}, and therefore $f_n(t_{n+1})< 1/4+1/8<1/2$. Thus $f_{n+1} \geq f_n$ follows from $f_{n+1}=f_n+\min\{\eps,d(\,\cdot\,,X\setminus G)\}$. Proving that $f_{n+1}|_E\leq f_n|_E$ is analogous. This completes the induction.

By condition \eqref{projections:p:4}, the sequence $\{f_n\}$ converges uniformly to a continuous function $f \colon X \to [0,1]$ and $d(f,f_0)\leq 1/8$. Since $f_{n+1}|_{T_n\cup\{x,y\}} = f_n|_{T_n\cup\{x,y\}}$ for all $n$ by \eqref{projections:p:1} and \eqref{projections:p:3}, we have $f|_{T_n\cup\{x,y\}} = f_n|_{T_n\cup\{x,y\}}$. In particular, $f$ is injective on $T$ by \eqref{projections:p:2}.

Let $z\in X\setminus\{x\}$. If $z\notin F$ then we obtain $f(z)\geq f_0(z)-1/8\geq 1/4-1/8>0$, otherwise $z\in F\setminus \{x\}$ implies $f(z)\geq f_0(z)>0$ by \eqref{projections:p:5}. Thus $f(z)>0$ in both cases. Analogically, $f(z)\leq f_0(z)+1/8\leq 3/4+1/8<1$ for every $z\in X\setminus E$ and  $f(z)\leq f_0(z)<1$ for every $z\in E\setminus \{y\}$. This finishes the construction and verification of desired properties of $f \colon X \to [0,1]$.

 Finally, assume that $X$ is a Peano continuum. It follows that $X$ has a countable base for the topology $\mathcal{U}$ formed by nonempty connected open sets. By choosing $T$ such that it intersects every $U \in \mathcal{U}$ in at least two points (which is possible since $\mathcal{U}$ is countable and every $U$ is infinite), the injectivity of $f$ on $T$ guarantees that $|f(U)| > 1$ for every $U \in \mathcal{U}$. Because $U$ is connected and $f$ is continuous, $f(U)$ is a connected set, hence, necessarily a nondegenerate interval. Thus, $f(U)$ has a nonempty interior.

\end{proof}}

    \begin{lem}\label{construction}
       Let $X$ be a nondegenerate Peano continuum and let the points $x_1,x_2 \in X$  not be local cut points (they do not have to be distinct). Then there is a continuous map $h:X\to X$ that is mixing and satisfies $h^{-1}(x_i)=\{x_i\}$ for $i=1,2$ (so, in particular, $h$ is not exact).
   \end{lem}
   \begin{proof}
   We will construct $h$ as the composition $h=g\circ f$ of maps $f\colon X\to [0,1]$ and $g\colon [0,1]\to X$. Choose any point $y\in X$, $x_1\ne y\ne x_2$. We apply Lemma \ref{projections} and obtain a continuous map $f\colon X\to [0,1]$ that maps any nonempty open set onto a set with nonempty interior, $f^{-1}(0)=\{x_1\}$. Furthermore, $f^{-1}(1)=\{x_2\}$ if $x_1\ne x_2$, or $f^{-1}(1)=\{y\}$ if $x_1= x_2$ (and then $f^{-1}(0)=\{x_2\}=\{x_1\}$ in this case).
   
    In what follows, we will present the construction for the case $x_1\ne x_2$. The construction for the case $x_1=x_2$ is easier, and the details are left to the reader. Roughly speaking, the construction for that case goes for $x_1$ and $y$ but without additional concerns about how $g$ behaves around $g^{-1}(y)$. In the considered case we have $d(x_1,x_2)>1$.
   We will construct $g$ by induction, providing its consecutive approximations $g_n$ with respect to $n\geq 0$. For each $n$ we will provide:
\begin{itemize}
    \item $\H_n=\F_n\cup\{A_n,A_n'\}$: a finite cover of $X$ by Peano continua with nonempty interiors{, $A_n \cap A'_n=\emptyset $} and with $| \F_0|=1$,
    \item $a_n<b_n\in(0,1)$,
    \item $g_n:I\to \bigcup\F_n$ continuous surjection,
    \item $S_n$: a finite subset of $[a_n,b_n]$. 
\end{itemize}

such that for each $n\in \N_0$: 

\begin{enumerate}
    \item\label{prop2} $d(g_{n-1},g_{n})\leq 2^{-n+1}$ when $n>1$,
    \item\label{prop1} $\mesh\F_n\leq 2^{-n-1}$ when $n>0$ and $\max\{\diam A_n, \diam A_n'\}< 2^{-n-2}$,    
    \item\label{prop3a} $x_1\in A_n\setminus(\bigcup\F_n)$, $x_2\in A_n'\setminus(\bigcup\F_n)$ 
     \item\label{prop7} for every $F\in\F_{n-1}$ there is $\F\subset\F_{n}$ satisfying $\bigcup\F=F$ for $n>0$,
    \item\label{prop3b} $a_{n}\leq a_{n-1}/2$, $b_{n}\geq (1+b_{n-1})/2$ for $n>0$, 
    \item\label{prop8} $[a_{n},b_{n}]=f(\bigcup\F_{n-1})\supset [a_{n-1}/2,\,(b_{n-1}+1)/2]$ for $n>0$,
    \item\label{prop4} $g_n([0,a_n])\subset A_n$, $g_n([b_n,1])\subset A_n'$,
    \item\label{prop6} $g_n([a_i,b_i])=\bigcup\F_i$ for every $i\leq n$,
    \item\label{prop9} ${(}f\circ g_n{)}(S_n)=S_n$ (and thus points in $g_n(S_n)$ are periodic under $ g_n\circ f$),
    \item\label{prop10} $S_{n-1}\subset S_{n}$ and $g_{n}|_{S_{n-1}}=g_{n-1}{|_{S_{n-1}}}$ for $n>0$,
    \item\label{prop11} $f(F)\cap S_{n}\ne \emptyset$ for every $F\in \F_{n-1}$ and every $n>0$.
    \item\label{prop5} $(g_n\circ f)^{2i}(F)\cap (g_n\circ f)^{2i+1}(F) \supset \bigcup \F_i$ for every $F\in \F_i$ and $0\leq i\leq n$,
    \item\label{prop13}  for every ${1\leq}i \leq n$ and  $F \in \F_i$ there is {$F' \in \F_{i-1}$ }such that $(g_n \circ f)(F) \supset F'$ 
\end{enumerate}

Note that since $d(g_n,g_{n+1})\leq 2^{-n+1}$ by \eqref{prop2}, the construction produces a Cauchy sequence and, consequently, $g:=\lim_{n\to \infty} g_n$ is a well-defined continuous map. This map is the core of the construction. Let us first explain how the inductive construction is performed.

Fix sufficiently small $0<\varepsilon< 2^{-3}$ so that $f(B(x_1,\varepsilon))\subset [0,1/8]$ and $f(B(x_2,\varepsilon))\subset [7/8,\,1]$. {Note that $x_1$ is a non-cut point of $X$ and thus there are Peano continua $A_0,B\subset X$ such that $x_1\in X\setminus B\subset A_{0}\subset B(x_1,\varepsilon)$ and $\cl{\Int B}=B$ by Lemma \ref{noncutproperty}}. {Then $x_2 \in X\setminus A_0\subset B$ as $x_2\notin B(x_1,{\eps})$, implying $B$ is a connected neighborhood of $x_2$, and therefore $x_2$ is a non-cut point of $B$ since it is not a local cut point of $X$. Hence there are Peano continua $A_0',B'\subset B$ such that $x_2\in B\setminus B'\subset A_{0}'\subset B(x_2,\varepsilon)$ and $\cl{\Int B'}=B'$ by Lemma \ref{noncutproperty}.}
Let $\F_0:=\{B'\}$, $a_0:=1/4$, $b_0:=3/4$, $S_0:=\emptyset$. Note that $A_0\cap B'\ne\emptyset\ne A_0'\cap B'$ since $A_0\cup B'\cup A_0'=X$, all $A_0,B',A_0',X$ are connected and $A_0\cap A_0'=\emptyset$. Thus, we may choose points $a\in A_0\cap B'$ and $b\in A_0'\cap B'$.
By Theorem~\ref{nadler:8.19} there is a continuous surjection {$g_0:[0,1]\to X$ satisfying $g_0(1/4)=a$, $g_0(3/4)=b$, $g_0(0)=x_1$, $g_0(1)=x_2$ and such that $g_0([1/4,3/4])=B'$, $g_0([0,1/4])
=A_0$ and $g_0([3/4,1])
=A_0'$.}
{It is easy to check that these satisfy the induction hypothesis.}
Assume that we have already defined $S_n, g_n$, $\H_n$, $a_n$, $b_n$ and we will find $S_{n+1},g_{n+1}$, $\H_{n+1}$, $a_{n+1}$, $b_{n+1}$.

Note that $f(\bigcup\F_n)$ is connected and closed, as $\bigcup\F_n$ is a continuous image of $I$ under $g_n$. Denote $[a_{n+1},b_{n+1}]:=f(\bigcup\F_n)$ and observe that $a_{n+1}<a_n<b_n<b_{n+1}$ by \eqref{prop8}.
Similarly to the initial step, take sufficiently small $0<\varepsilon< 2^{-n-{4}}$ such that $f(B(x_1,\varepsilon))\subset [0,a_{n+1}/2]$ and $f(B(x_2,\varepsilon))\subset [(1-b_{n+1})/2,\,1]$. Decreasing $\varepsilon$ when necessary, we may assume that $B(x_1,\varepsilon)\subset A_n\setminus(\bigcup\F_n)=X\setminus(\bigcup\F_n\cup A_n')$ and, similarly, $B(x_2,\varepsilon)\subset A_n'\setminus(\bigcup\F_n)$. {Note that $x_1$ is a non-cut point of $A_n$ and $x_2$ is a non-cut point of $A_n'$, their respective connected neighborhoods, since we assume that none of $x_1,$, $x_2$ is a local cut point of $X$. Thus there are Peano continua $A_{n+1},B\subset A_n$, resp. $A_{n+1}',B'\subset A_n'$,  such that $x_1\in A_n\setminus B\subset A_{n+1}\subset B(x_1,\varepsilon)$, resp. $x_2\in A_n'\setminus B'\subset A_{n+1}'\subset B(x_2,\varepsilon)$, and $\cl{\Int B}=B$, $\cl{\Int B'}=B'$ by Lemma \ref{noncutproperty}.}

Recall that $g_n(a_n)\in A_n\cap(\bigcup \F_n)$ by \eqref{prop4} and \eqref{prop6}, hence $g_n(a_n)\in A_n\setminus A_{n+1}\subset B$ since $A_{n+1}\subset B(x_1,\varepsilon)\subset A_n\setminus(\bigcup\F_n)$. By \eqref{prop1}, we have $2^{-n-2}-\diam B {\geq}2^{-n-2}-\diam A_n>0$. Therefore, by  Theorem~\ref{finecovers} applied to the Peano continuum $\bigcup\F_n$ there exists a Peano continuum $K$ such that $g_n(a_n)\in K$, $K\subset\bigcup\F_n$, $\cl{\Int K}=K$, and $\diam K\leq 2^{-n-2}-\diam B$.

Put $C:=B\cup K$ and note that $\diam C<2^{-n-2}$.
We may repeat the above argument for $g_n(b_n)\in B'$ and find a Peano continuum $C'$ such that $B'\subset C'$, $\diam C'< 2^{-n-2}$, and $C'\cap (\bigcup\F_n)$ contains a nonempty open set.

By Theorem \ref{finecovers}  applied consecutively for each  Peano continuum in $\F_n$, there is a finite family of Peano continua $\F_{n+1}'$ such that $\mesh\F_{n+1}'\leq 2^{-n-2}$, $\cl{\Int F}=F$  for every $F\in\F_{n+1}'$ and for every $F\in\F_n$, there is $\F\subset\F_{n+1}'$ such that $\bigcup\F=F$. 
Put $\F_{n+1}:=\F_{n+1}'\cup\{C,C'\}$.
Note that every $F\in\F_{n+1}$ contains a nonempty open subset of $\bigcup\F_n$ and thus $f(F)$ contains a nondegenerate subinterval of $[a_{n+1},b_{n+1}]=f(\bigcup\F_n)$.

Let us fix an increasing sequence $t_0<t_1<\dots<t_k\in[0,1]$ with the following properties:
\begin{itemize}
    \item $t_0=a_{n+1}$, $t_1=a_n$, $t_{k-1}=b_n$ and $t_k=b_{n+1}$,
    \item the sequence contains both $a_i,b_i$ for every $i<n$,
    \item the sequence contains both endpoints of $f(F)$ for every $F\in\F_i$ and every $i\leq n+1$ except the case $i=n+1$ and $F\in \{C,C'\}$ (recall that any such $f(F)$ is necessarily a nondegenerate closed interval that is a subset of $[a_{n+1},b_{n+1}]$),
    \item the sequence contains the endpoint of $f(C)$, resp. of $f(C')$, that lies in $[a_{n+1},b_{n+1}]$ (recall that $f(C)=[r,s]$ for some $r\leq a_{n+1}/2<a_{n+1} {\leq }s$ and $f(C')=[r',s']$ for some $r '{\leq}  b_{n+1}<(1+b_{n+1})/2\leq s'$),
    \item $\diam (g_n([t_{i-1},t_i]))< 2^{-n-1}$ for every $2\leq i\leq k-1$.
\end{itemize}

Note that for both $i=1,k$ we have $\diam (g_n([t_{i-1},t_i]))\leq \max\{A_n,A_n'\}<2^{-n-2}$ by \eqref{prop4} and \eqref{prop1}, therefore $\diam (g_n([t_{i-1},t_i]))< 2^{-n-1}$ for every $1\leq i\leq k$.
For each $1\leq i\leq k$ choose a Peano continuum $H_i\subset X$ as follows:
\begin{enumerate}[(i)]
    \item If $[t_{i-1},t_i]\subset [a_n,b_n]$, i.e. if $2\leq i\leq k-1$, find the smallest $0\leq j\leq n$ such that $[t_{i-1},t_i]\subset [a_j,b_j]$. By 
    \eqref{prop6}, $g_n([t_{i-1},t_i])\subset\bigcup \F_{j}$ and thus by \eqref{prop7}, there is $F\in\F_n$ such that $F\subset\bigcup \F_{j}$ and $F\cap g_n([t_{i-1},t_i])\ne \emptyset$. Put $H_i:=F\cup g_n([t_{i-1},t_i])$.    
    \item Choose $F\in\F_n$ such that $g_n(a_n)\in F$ and put $H_1:=F\cup C$. Note that $H_1$ is a Peano continuum since it is connected, as $g_n(a_n)\in F\cap C$. Similarly, let $H_k:= F'\cup C'$ for some $F'\in\F_n$ satisfying $g_n(b_n)\in F'$.
\end{enumerate}

By the above construction,
 for every $1\leq i\leq k$
\begin{eqnarray}
\diam\label{con:i:gn} H_i&\leq&  \max\{\diam g_n([t_{i-1},t_i]),\diam C, \diam C'\}
+ \mesh\F_n\nonumber\\
&\leq& 2^{-n-1} + 2^{-n-1} =  2^{-n}\label{eq:diamHi}
\end{eqnarray}
and we also have
\begin{eqnarray}
 g_n([t_{i-1},t_i])\subset H_i \label{eq:Hisupsgn}
\end{eqnarray}

since for $i=1$ (similar argument applies for $i=k$) we have
\[
g_n([t_{0},t_1])\subset A_n\cap \bigcup\F_n\subset B\subset C\subset H_1.
\]

{
Next, we choose points $x_F$, $x'_F$  for every $F \in \mathcal{F}_{n}$ successively.
Proceeding through the sets $F \in \mathcal{F}_{n}$ one by one, we first choose a point $x_F \in F$ such that
\[
f(x_F)\in f(F)\setminus (S_n\cup\{t_0,t_1,\dots,t_k\})
\]

and such that $f(x_F) \neq (f\circ g_n)^i(z)$ for all $ 0 \leq i \leq 2n $,  each $z$ in $\{f(x_{F'}'),f(x_{F'})\}$ for all previously processed $F' \in \mathcal{F}_{{n}}$ and also for each $z$  in $\{a_{n+1},b_{n+1}\}$. This is easily seen to be possible since any such $f(F)$ is a nondegenerate interval and we are avoiding only finitely many choices. 

There is a unique $1\leq i\leq k$ such that $f(x_F)\in (t_{i-1},t_i)$ since $f(F)\subset f(\bigcup \F_n)=[a_{n+1},b_{n+1}]$. By \eqref{prop5} we have $(g_n\circ f)^{2n}(H_i)\supset \bigcup\F_n\ni x_F$ since $H_i$ contains a member of $\F_n$. Hence we can select $x_F'\in H_i$ satisfying $(g_n\circ f)^{2n}(x_F')=x_F$.

We claim that for any $F' \in \mathcal{F}_{{n}}$ we have $z \neq f((g_n\circ f)^{i}(x_{F'}'))$ for every $z\in \{a_{n+1},b_{n+1}\}\cup\{f(x_F);\,F \in \mathcal{F}_{{n}}\setminus\{F'\}\}$ and for every $0 \leq i \leq 2n $.
Indeed, if $z=f(x_F)$ for some $F\in\F_n$ that was processed after $F'$, this follows directly from the choice of $x_F$.
Thus assume $z\in \{a_{n+1},b_{n+1}\}$ or $z=f(x_F)$ for some $F\in\F_n$ that was processed before $F'$. Note that
$f((g_n\circ f)^{2n-i}((g_n\circ f)^{i}(x_{F'}'))) = f(x_{F'})$, while $(f\circ g_n)^{2n-i}(z)$ cannot be equal to $f(x_{F'})$
by the choice of $x_{F'}$, and thus $z$ cannot be equal to $f((g_n\circ f)^{i}(x_{F'}'))$.
}

\begin{equation}\label{eq:S_n+1}
    S_{n+1}:= S_n\cup \{f((g_n\circ f)^{i}(x_F'));\,0\leq i\leq 2n,F\in\F_{n}\}.
\end{equation}

Further, choose any points $a\in A_{{n+1}}\cap B$, $b\in A_{{n+1}}'\cap B'$ and  define the function $g_{n+1}$ on the set $S_{n+1}\cup\{t_1,\dots,t_{k-1}\}\cup[0,a_{n+1}]\cup[b_{n+1},1]$ as:

\begin{equation}\label{eq:g_n+1}
    g_{n+1}(t) := \begin{cases}
        x_F', &t=f(x_F),\,F\in\F_n; \\
        a, & t\leq a_{n+1}; \\
        b, & t\geq b_{n+1}; \\
        g_n(t), & \text{otherwise}.
       \end{cases}
\end{equation}

{
Note that for every $F \in \F_{n}$ and every $0\leq i\leq 2n$ we have $f((g_n\circ f)^{i}(x_F'))\in f(\bigcup\F_n) = [a_{n+1},b_{n+1}] $ and $f((g_n\circ f)^{i}(x_F'))\neq z$ for $z\in\{a_{n+1},b_{n+1}\}\cup\{f(x_{F'});\,F' \in \mathcal{F}_{{n}}\setminus\{F\}\}$ so $g_{n+1}$ is well defined. Further, if $f((g_n\circ f)^{i}(x_F'))\neq f(x_F)$ then  $g_{n+1}(f((g_n\circ f)^{i}(x_F')))=g_n(f((g_n\circ f)^{i}(x_F')))$ and $g_{n+1}(f(x_F))=x_F'$.}
 Finally, for every $1\leq i\leq k$, we apply Theorem~\ref{nadler:8.19} to extend  the map $$g_{n+1}|_{[t_{i-1},t_i]\cap (S_{n+1}\cup\{t_0,t_1,\dots,t_k\})}$$ with values in $H_i$ to a continuous surjection $g_{n+1}|_{[t_{i-1},t_i]}:[t_{i-1},t_i]\to H_i$.
It is easy to observe that the resulting map $g_{n+1}\colon 
I \to X$ is well-defined and continuous.

To justify that the inductive hypotheses for $n+1$ are indeed satisfied, note first that  { \eqref{prop1}, \eqref{prop3a}, \eqref{prop7}, \eqref{prop3b}, \eqref{prop4}, \eqref{prop6} follow directly from the construction. Further, note that by \eqref{eq:S_n+1} $S_n\subset S_{n+1}$ and moreover, for every $F\in\F_n$ we have $f((g_n\circ f)^{2n}(x_F'))\in S_{n+1}\cap {f(F)}$ since $(g_n\circ f)^{2n}(x_F')=x_F\in F$, in particular, \eqref{prop11} is true. For the second part of \eqref{prop10}, note that $S_n\subset[a_n,b_n]\subset (a_{n+1},b_{n+1})$ and also $f(x_F)\notin S_n$ for every $F\in \F_n$ by our choices, hence $g_{n+1}|_{S_n}=g_n|_{S_n}$ follows from \eqref{eq:g_n+1}. 

We are going to show that \eqref{prop9}, \eqref{prop2}, \eqref{prop8}, \eqref{prop5}, and \eqref{prop13} also hold, which will complete the induction.

To verify \eqref{prop9}, let $s\in S_{n+1}$, we wish to show that $(f \circ g_{n+1})(s) \in S_{n+1}$. If $s\in S_n$, the claim follows from \eqref{prop10}, which we have already proved. Further, if $s=f(x_F)$ for some $F\in\F_n$ then  $(f\circ g_{n+1})(s)= f  (g_{n+1}(f(x_F)))=f(x_F')=f((g_n\circ f)^{0}(x_F')) \in S_{n+1}$. In the remaining case there is some $F \in \F_n$ and $0 \leq i \leq 2n$ such that $s=f((g_n\circ f)^{i}(x_F'))$, since $s \in S_{n+1} \setminus S_n$. Note that necessarily $i<2n$ since $(g_n\circ f)^{2n}(x_F')=x_F$. Thus $(f\circ g_{n+1})(f((g_n\circ f)^{i}(x_F')))=(f\circ g_n)(f((g_n\circ f)^{i}(x_F')))=f((g_n\circ f)^{i+1}(x_F'))\in S_{n+1}$.} {Analogously, pick $s \in S_{n+1}$. As before, the case when $s \in S_n$ follows imminently, so assume $s \in S_{n+1} \setminus S_n$, meaning $s=f((g_n\circ f)^{i}(x_F'))$ for some $F \in \F_n$ and $0 \leq i \leq 2n$. If $i>0$, then we get $s=(f \circ g_n \circ f)(g_n \circ f)^{i-1}((x_F'))=(f \circ g_{n+1})(f((g_n \circ f)^{i-1}((x_F'))) \in (f\circ g_{n+1} )(S_{n+1})$. If $i=0$ then $s=f(x_F')=f((g_{n+1}\circ f)(x_F))=f((g_{n+1}\circ f)((g_n\circ f)^{2n}(x_F')))= (f \circ g_{n+1})(f((g_n\circ f)^{2n}(x_F')) \in (f\circ g_{n+1})(S_{n+1})$.}

To verify \eqref{prop8}, observe that $\bigcup\F_{n+1}=\bigcup\F_n\cup C\cup C'$ is connected, as both $C,C'$ intersect $\bigcup\F_n$. Hence, $f(\bigcup\F_{n+1})$ is connected as well. Thus, it suffices to check that $f(B)\cap[0,a_{n+1}/2]\ne\emptyset$ and the analogous condition for $f(B')$, where $B,B'$ are sets from the construction of $A_{n+1}$ and $A_{n+1}'$. Recall that $B\cap A_{n+1}\ne\emptyset$ by the connectedness of $A_n$ and $f(A_{n+1})\subset f(B(x_1,\varepsilon))\subset[0,a_{n+1}/2]$, so indeed \eqref{prop8} holds.

To verify \eqref{prop2}, assume that $n\geq 1$ and recall that (in the notation from the construction of $g_{n+1}$) for every $1\leq i\leq k$ there is $H_i$ satisfying $H_i\supset g_{n}([t_{i-1},t_i])$ by \ref{eq:Hisupsgn}, $\diam H_i\leq 2^{-n}$ by \eqref{eq:diamHi} and $g_{n+1}([t_{i-1},t_i])=H_i$ by the construction. Hence, for every $1\leq i\leq k$ and every $t\in[t_{i-1},t_i]$, we have $g_n(t),g_{n+1}(t)\in H_i$, implying $d(g_n(t),g_{n+1}(t))\leq\diam H_i\leq 2^{-n}$.

To verify \eqref{prop13}, first fix $0\leq i\leq n$ and $F\in \F_i$. The choice of $a_{n+1},b_{n+1}$ and \eqref{prop7} gives us $f(F)\subset[a_{n+1},b_{n+1}]$. Further, the choice of $t_0,\dots,t_k$ gives that there are $0\leq m<l\leq k$ satisfying $f(F)=[t_m,t_l]$.
Moreover, it follows from \ref{eq:Hisupsgn} and the construction that $g_{n+1}([t_{j-1},t_j])\supset g_{n}([t_{j-1},t_j])$ holds for every $1\leq j\leq k$. These arguments combined finally prove that 
\begin{equation}\label{magnifying}
   (g_{n+1}\circ f)(F) \supset (g_n\circ f)(F).
\end{equation}

We are prepared to verify \eqref{prop13} now, so let $1\leq i\leq n+1$ and $F\in \F_i$. If $i\leq n$, then \eqref{magnifying} and the condition \eqref{prop13} from the earlier step of induction indicate that there is $F' \in \F'_{i-1}$ such that $(g_{n+1} \circ f)(F)\supset(g_n \circ f)(F) \supset F'$. Assume next that $i=n+1$ and recall that by the definition, there is $1\leq m\leq k$ such that $[t_{m-1},t_m]\subset f(F)$, because the endpoints of $f(F)$ are within the sequence $t_j$.
Thus, by the construction, $g_{n+1}(f(F))\supset H_m$ contains a member of $\F_n$ as desired.

Lastly, we verify \eqref{prop5}. First, note that for any $1 \leq j \leq n+1$, we have the following: 
\begin{equation}
\label{eq:prop12}
    (g_{n+1} \circ f)(\bigcup \F_{j-1}) = g_{n+1}(f(\bigcup \F_{j-1} )) \stackrel{\eqref{prop8}}{=} g_{n+1}([a_j,b_j]) \stackrel{\eqref{prop6}}{=}\bigcup\F_j.
\end{equation}

Take any $0 \leq i \leq n+1$ and any $F \in \F_i$, by using \eqref{prop13} inductively, we get that there is $F' \in \F_{0}$ such that $(g_{n+1} \circ f)^{i}(F) \supset F'$. As family $\F_0$ contains one element, we get that $(g_{n+1} \circ f)^{i}(F) \supset F'=\bigcup \F_0$. By the inductive use of \eqref{eq:prop12}, we finally obtain $(g_{n+1}\circ f)^{2i}(F)  \supset (g_{n+1}\circ f)^{i}(\bigcup \F_0) \supset \cdots \supset\bigcup\F_i$. 

Consequently,
\begin{align*}
    (g_{n+1} \circ f)^{2i+1}(F) & \supset (g_{n+1} \circ f) ( \bigcup \F_i) \supset (g_{n+1} \circ f) ( \bigcup \F_{\min\{i,n\}}) \\& \stackrel{\eqref{prop6}}{\supset} \bigcup \F_{\min\{i+1,n+1\}}) \supset \bigcup \F_i.
\end{align*}

This finishes the induction step.

 As we said before, by \eqref{prop2} there is a continuous map $g:=\lim_{n\to \infty} g_n$.
It follows from \eqref{prop10} that $g|_{S_n}=g_n|_{S_n}$ for every $n\in\N$, and it follows from \eqref{prop9} that $g_n\colon S_n\to g_n(S_n)$ and $f\colon g_n(S_n)\to S_n$ are bijections. Therefore, every point of $g(\bigcup_{n\in\N} S_n)=\bigcup_{n\in\N} g(S_n)$ is a periodic point of $g\circ f$ by \eqref{prop9}. Further, $g(\bigcup_{n\in\N} S_n)$ is dense in $X$ by \eqref{prop1} and \eqref{prop11}. Therefore, periodic points of $g\circ f$ are dense in $X$.

Let $ U,V\subset X$ be nonempty and open. Then $U\setminus\{x_1,x_2\}$, $V\setminus\{x_1,x_2\}$ are nonempty and open as well. Thus, there exist $n\in \N$ and $E,F\in\F_n$ such that $E\subset U$ and $F\subset V$. 
Fix any $k\geq n$ and any $m\geq 2n$. By \eqref{prop5} and \eqref{eq:prop12} {used $m-2n$ times} we have 

{
$$
(g_k \circ f)^m(E)=(g_k \circ f)^{m-2n}((g_k \circ f)^{2n}(E))  \stackrel{\eqref{prop5}}\supset (g_k \circ f)^{m-2n}(\bigcup\F_n) \stackrel{\eqref{eq:prop12}}\supset \bigcup\F_n\supset F.
$$
}
Since the relation holds for any $k>n$, we get  $(g\circ f)^{m}(U)\cap V\ne \emptyset$ for every $m\geq 2n$, and therefore $g\circ f$ is mixing.

Finally, we verify that $g^{-1}(x_1)=\{0\}$ and $g^{-1}(x_2)=\{1\}$. It follows easily by combining \eqref{prop1}, \eqref{prop3a} and \eqref{prop4} that $g(0)=x_1$, $g(1)=x_2$. Next, let $t\in (0,1)$ be arbitrary. There exists $n\in\N$ such that $t\in[a_n,b_n]$ since $\bigcup_{n\in\N}[a_n,b_n]=(0,1)$ by \eqref{prop3b}. By \eqref{prop6}, we have that $g_k([a_n,b_n])=\bigcup\F_n$ for every $k\geq n$, and hence $g([a_n,b_n])\subset\bigcup\F_n$ since $\bigcup\F_n$ is closed. By \eqref{prop3a}, $x_1,x_2\notin \bigcup\F_n$, and therefore, $g(t)\notin \{x_1,x_2\}$, which concludes the proof.
    \end{proof}
   \begin{proof}[Proof of Theorem~\ref{thm:pure}]

       By Lemma \ref{existencenonloc} there is $x\in X$ that satisfies $\eqref{notlocalcutpoint}$ or \eqref{localcutpointordertwo} from the statement Lemma \ref{existencenonloc}.
       If $x$ satisfies \eqref{notlocalcutpoint}, then applying Lemma \ref{construction} for $x_1=x_2=x$ we obtain a mixing map $h\colon X\to X$ which is not exact, finishing the proof. 
       
       Assume that $x$ satisfies \eqref{localcutpointordertwo} instead. Let $\widehat{X}$ together with the map $\phi\colon \widehat{X}\to X $ be provided by Lemma~\ref{lifting} for $x$. 
       Then $x_1,x_2\in \phi^{-1}(x)$ are endpoints of the Peano continuum $\widehat{X}$, in particular neither of them is a local cut point of $\widehat{X}$.
       Lemma \ref{construction} provides a mixing map $\hat{f}:\widehat{X}\to\widehat{X}$ such that $\hat{f}(x_i)=x_i$ 
       and $\hat{f}^{-1}(x_i)=\{x_i\}$. Since $\pi^{-1}$ is well defined homeomorphism on $X\setminus \{x\}$, the map $\hat{f}$ induces the unique map $f\colon X\to X$ such that $f\circ \phi=\phi \circ \hat{f}$. Mixing is preserved by factor maps, hence $f$ is mixing, but by the definition $f^{-1}(x)=\{x\}$, hence it is not exact. The proof is complete.
   \end{proof} 

The following lemma is standard. Proof is left to the reader.
\begin{lem}\label{lem:changed_order_mix}
    Let $X,Y$ be compact metric spaces, and $f\colon X\to Y$, $g\colon Y\to X$ continuous, and $f$ surjective. Then, if $g\circ f$ is mixing, so is $f\circ g$.
\end{lem}

\begin{rem}
While constructing the pure mixing map in Theorem~\ref{thm:pure}, as a middle step we constructed continuous maps $f\colon X\to [0,1]$, $g\colon [0,1]\to X$ such that $g\circ f$ is mixing. Since $f$ is surjective, it easily follows that $f\circ g =[0,1]\to [0,1]$ is mixing too (see Lemma~\ref{lem:changed_order_mix}). Moreover, $(f\circ g)^{-1}(0)=\{0\}$, hence, $f\circ g$ is pure mixing.
\end{rem}

It is clear that if $f$ is a mixing map on $[0,1]$, it cannot be invertible. On the other hand, there are several mixing diffeomorphisms on the torus (e.g. see \cite{Brin}; and clearly, homeomorphisms cannot be exact) while maps constructed by our method are never invertible. 
This leads to the following question:

\begin{que}
    Is it possible to characterize Peano continua admitting mixing homeomorphisms?
\end{que}

\section{Final remark on entropy}

Maps in the main construction appear as compositions of some auxiliary mappings. So let us first recall what is known about entropy of compositions and compositions in reversed order.
If $f,g\colon X\to X$, then Theorem~A in \cite{Kolyada} states the following:
\begin{thm}\label{thm:A:Kolyada}
For any continuous maps $f,g\colon X\to X$, we have $h_{top}(f\circ g)=h_{top}(g\circ f)$
\end{thm}

The above result can be easily extended to the following.

\begin{cor}\label{cor:ent_order}
For any continuous maps $f\colon X\to Y$ , $g\colon Y\to X$, we have $h_{top}(f\circ g)=h_{top}(g\circ f)$
\end{cor}
\begin{proof}
Put $Z=(X\cup Y)\times\{0,1\}$ and define maps $F,G\colon Z\to Z$ by:
$F(y,a)=(y,1)$ for $y\in Y$, $a=0,1$; $F(x,0)=(f(x),0)$ and $F(x,1)=(x,1)$ for $x\in X$.
$G(x,a)=(x,1)$ for $x\in X$, $a=0,1$; $G(y,0)=(g(y),0)$ and $G(y,1)=(y,1)$ for $y\in Y$.

Then $(G\circ F)(x,0)=((g\circ f)(x),0)$ for $x\in X$; $(G\circ F)(x,1)=(x,1)$ for $x\in X$; $(G\circ F)(y,1)=(y,1)=(G\circ F)(y,0)$ for $y\in Y$. Similarly
$(F\circ G)(y,0)=((f\circ g)(y),0)$ for $y\in Y$; $(F\circ G)(y,1)=(y,1)$ for $y\in Y$; $(G\circ F)(x,1)=(x,1)=(G\circ F)(x,0)$ for $x\in X$.

It is clear that $h_{top}(F\circ G)=h_{top}(f\circ g)$ and $h_{top}(G\circ F)=h_{top}(g\circ f)$, so the result follows from Theorem~\ref{thm:A:Kolyada}.
\end{proof}

The map $h$ in Theorem~\ref{thm:pure} is the composition $h=g\circ f$ where $f\colon X\to [0,1]$ and $g\colon [0,1]\to X$. The interval map $h=f\circ g$ is pure mixing, which, by Lemma~\ref{lem:changed_order_mix}, together with the results of \cite{HK} (see also \cite{HKO}), implies that $h_{top}(\hat h)\geq \log(3)/2$, and hence, by Corollary~\ref{cor:ent_order}, we have that $h_{top}(h)\geq \log(3)/2$. We have just proven the following:
    \begin{rem}\label{rem:pure}
    The pure mixing map constructed in Theorem~\ref{thm:pure}
has topological entropy of at least $\log(3)/2$. 
    \end{rem}

In fact, by the construction in the proof of Lemma~\ref{construction}, we have that $h^{2n}(F)\supset \cup \F_n$ for every $F\in \F_n$ where $\mesh \F_n<2^{-n}$ and $\F_n$, together with two more elements of small diameter, is a cover of the continuum $X$. In fact, we can (and in practice do) require that the diameters of elements of $\F_n$ decrease very rapidly, which in turn means that $\F_n$ consists of many more than just $2^n$ elements. Therefore, in practice, the entropy of $h$ will be very large or even infinite. We can clearly increase the entropy within the construction, and the possibility of decreasing it is not obvious.

    \begin{cor}\label{cor:pure}
        Every nondegenerate Peano continuum admits a pure mixing selfmap of infinite topological entropy. 
    \end{cor}

On the other hand, \v{S}pitalsk\'y proved in \cite{Vlad} that on various dendrites, the entropy of the exact map can be arbitrarily small.
The recent \cite{KOT} proves that in Gehman dendrite pure mixing maps can have arbitrarily small entropy. In particular, maps provided by Corollary~\ref{cor:pure} can have entropies quite far from the possible infimum. This motivates the following question, which concludes our paper (for more motivation and a brief history of research on similar topics, the reader is referred to the introduction of \cite{HKO} and references therein):
\begin{que}
    Let $X$ be a nondegenerate Peano continuum. 
    What is the infimum of entropy over family of: transitive, mixing, exact or pure mixing maps on $X$.
\end{que}

\section{Acknowledgements}

Research of K. Karasová was supported by the grant n. 129024 of the Charles University Grant Agency, by the grant GACR 24-10705S and by the Charles University project SVV-2025-260837.

Research of P. Oprocha was partially supported by the project No.~CZ.02.01.01/00/\break{}23\_021/0008759 supported by EU funds, through the Operational Programme Johannes Amos Comenius.

\bibliographystyle{amsplain}
\bibliography{references}
\end{document}